\documentclass[final,leqno]{siamltex}
\usepackage{amsmath}
\usepackage{amssymb}
\usepackage{graphicx}
\usepackage{tikz}
 \numberwithin{dummy}{section}

\newtheorem{algorithm}{Algorithm}

\newcommand{\bq}{{\bf q}}

\newcommand{\bx}{{\bf x}}

\newcommand{\be}{{\bf e}}
\newcommand{\bv}{{\bf v}}

\newcommand{\bpsi}{{\boldsymbol\psi}}

\def\T{{\mathcal T}}
\def\E{{\mathcal E}}

\def\pT{{\partial T}}
\def\l{{\langle}}
\def\r{{\rangle}}

\def\T{{\mathcal T}}
\def\E{{\mathcal E}}

\def\b#1{{\bf #1}}

\def\bn{{\bf n}}
\def\bq{{\bf q}}

\def\3bar{{|\hspace{-.02in}|\hspace{-.02in}|}}

\def\div{\operatorname{div}}

\def\ad#1{\begin{aligned}#1\end{aligned}}  \def\b#1{\mathbf{#1}} 
\def\a#1{\begin{align*}#1\end{align*}} \def\an#1{\begin{align}#1\end{align}} 
 
\def\p#1{\begin{pmatrix}#1\end{pmatrix}}

\title{A mixed finite element method on polytopal mesh}
\author{Yanping Lin\thanks{Department of Applied Mathematics, Hong Kong Polytechnic University, Hung Hom, Hong
Kong, China (yanping.lin@polyu.edu.hk).}
\and Xiu Ye\thanks{Department of
Mathematics, University of Arkansas at Little Rock, Little Rock, AR
72204 (xxye@ualr.edu). This research was supported in part by
National Science Foundation Grant DMS-1620016.}
\and
Shangyou Zhang\thanks{Department of
Mathematical Sciences, University of Delaware, Newark, DE 19716 (szhang@udel.edu).}
}
\begin{document}  \baselineskip=16pt\parskip=10pt
\maketitle

\begin{abstract}
In this paper, we introduce new stable mixed finite elements of any order on polytopal mesh for solving second order elliptic problem. We establish  optimal order error estimates for velocity and super convergence for pressure. Numerical experiments are conducted for our mixed elements of different orders on 2D and 3D spaces that confirm the theory.
\end{abstract}


\pagestyle{myheadings}

\section{Introduction}
The considered model problem seeks a flux function $\bq=\bq(\bx)$
and a scalar function $u=u(\bx)$ defined in an open bounded
polygonal or polyhedral domain $\Omega\subset\mathbb{R}^d\; (d=2,3)$
satisfying
\begin{eqnarray}
a\bq+\nabla u&=&0\quad \mbox{in}\;\Omega,\label{mixed1}\\
\nabla\cdot \bq&=&f\quad\mbox{in}\;\Omega,\label{mixed2}\\
 u&=&-g \quad \mbox{on}\; \partial\Omega,\label{bc1}
\end{eqnarray}
where $a$ is a symmetric, uniformly positive
definite matrix on the domain $\Omega$.
 A weak formulation for
(\ref{mixed1})-(\ref{bc1}) seeks $\bq\in H(\div,\Omega)$ and $u\in
L^2(\Omega)$ such that
\begin{eqnarray}
(\alpha\bq,\bv)-(\nabla\cdot\bv,u)&=&\langle g\bv\cdot\bn\rangle_{\partial\Omega}
\quad\forall\bv\in H(\div,\Omega),\label{w-mix1}\\
(\nabla\cdot\bq,w)&=&(f,w)\quad\forall w\in L^2(\Omega).\label{w-mix2}
\end{eqnarray}
Here $L^2(\Omega)$ is the standard space of square integrable
functions on $\Omega$, $\nabla\cdot\bv$ is the divergence of
vector-valued functions $\bv$ on $\Omega$, $H(\div,\Omega)$ is the
Sobolev space consisting of vector-valued functions $\bv$ such that
$\bv\in [L^2(\Omega)]^d$ and $\nabla\cdot\bv \in L^2(\Omega)$,
$(\cdot,\cdot)$ stands for the $L^2$-inner product in $L^2(\Omega)$,
and $\langle\cdot,\cdot\rangle_{\partial\Omega}$ is the inner
product in $L^2(\partial\Omega)$.

Finite element methods based on the weak formulation (\ref{w-mix1})-(\ref{w-mix2}) and finite dimensional subspaces of
$H(\div,\Omega)\times L^2(\Omega)$ with piecewise polynomials are
known as mixed finite element methods (MFEM). The mixed finite element methods have been intensively studied  \cite{ab,babuska,bf,bddf,bdm,rt, wang}
and many stable mixed finite elements have been developed such as  Raviart-Thomas (RT) and Brezzi-Douglas-Marini (BDM) elements. However most of the existing mixed elements are defined on triangle/rectangle in two dimensional space and tetrahedron/cuboid in  three dimensional space.

Construction of stable mixed finite elements on general polytopal mesh can be very challenging.
Recently, a lowest order mixed element on polytopal mesh was introduced in  \cite{cw} by using rational Wachspress coordinates.
The goal of this paper is to construct stable mixed elements of any order on polytopal mesh. Optimal convergence rate for velocity and superconvergence for pressure are obtained. Extensive numerical examples are tested for the new mixed finite elements of different degrees in two and three dimensions.

\section{Construction of a $H(\div,\Omega)$ Element}

Let ${\cal T}_h$ be a partition of the domain $\Omega$ consisting of
polygons in two dimension or polyhedra in three dimension satisfying
a set of conditions specified in \cite{wymix}. Denote by
${\cal E}_h$ the set of all edges/faces in ${\cal T}_h$, and let
${\cal E}_h^0={\cal E}_h\backslash\partial\Omega$ be the set of all
interior edges/faces. For simplicity, we will use term edge for edge/face without confusion.
Let $P_k(K)$ consist all the polynomials degree less or equal to $k$ defined on $T$.

The space $H(\div;\Omega)$ is defined as the set of vector-valued functions on $\Omega$ which,
together with their divergence, are square integrable; i.e.,
\[
H (\div; \Omega)=\left\{ \bv\in [L^2(\Omega)]^d:\; \nabla\cdot\bv \in L^2(\Omega)\right\}.
\]

For any $T\in\T_h$, we divide it in to a set of disjoint triangles/tetrahedra
   $T_i$ with $T=\cup T_i$.  We define $\Lambda_h(T)$  as
\begin{eqnarray}
\Lambda_k(T)=\{\bv\in H(\div;T):&&\ \bv|_{T_i}\in RT_k(T_i),\;\;\nabla\cdot\bv\in P_k(T)\},\label{lambda}
%&&\bv\cdot\bn|_e\in P_k(e),\;e\subset\pT\},\nonumber
\end{eqnarray}
where $RT_k(T_i)=\{[P_k(T_i)]^d \oplus \b x \sum_{|\alpha|=k} a_\alpha \b x^\alpha\}
   $ is the usual Raviart-Thomas element of order $k$.

Associated with the given mesh,  we introduce two finite element spaces
\begin{equation}\label{vhspace}
V_h=\{\bv\in H(\div;\Omega):\; \bv|_T\in \Lambda_k(T),\;\;  T\in \T_h\},
\end{equation}
and
\begin{equation}\label{whspace}
W_h=\{w\in L^2(\Omega):\; w|_T\in P_k(T),\;\;  T\in \T_h\}.
\end{equation}

\begin{lemma}
For the projection $\Pi_h$ defined in \eqref{p-j} below and
   for $\tau\in H(\div;\Omega)$ and $v\in P_k(T)$, we have
\begin{eqnarray}
(\nabla\cdot\tau,\;v)_T&=&(\nabla\cdot\Pi_h\tau,\;v)_T, \label{key2}\\
%-(\nabla\cdot\tau, \;v_0)_{\T_h}&=&(\Pi_h\tau, \;\nabla_w v)_{\T_h}\label{key1}\\
\|\Pi_h\tau-\tau\|&\le& Ch^{k+1}|\tau|_{k+1}.\label{key3}
\end{eqnarray}
\end{lemma}

\begin{proof}   We assume no additional inner vertex/edges is introduced in subdividing
   a polygon/polyhedron $T$ in to $n$ triangles/tetrahedrons $\{T_i\}$.
   That is,  we have precisely $n-1$ internal edges/triangles which separate $T$
    into $n$ parts.
  We limit the proof to 3D.  We need only omit the fourth equation in \eqref{p-j} to get
  a 2D proof.

\def\a#1{\begin{align*}#1\end{align*}} \def\p#1{\begin{pmatrix}#1\end{pmatrix}}
\def\an#1{\begin{align}#1\end{align}}\def\ad#1{\begin{aligned}#1\end{aligned}}

  On $n$ tetrahedrons,  a function of $\Lambda_k$ can be expressed as
 \an{\label{v-e}  \bv_h|_{T_{i_0}} = \sum_{i+j+l\le k}  \p{a_{1,ijl}\\a_{2,ijl}\\a_{3,ijl}} x^i y^j z^l
                + \sum_{i+j+l= k} \p{x\\y\\z} a_{4,ijl} x^i y^j z^l, \
    i_0=1,...n.  }
$\bv_h|_{T_{i_0}}$ is determined by
  \a{ \frac{n(k+1)(k+2)(k+3)}2 + \frac{ n(k+1)(k+2)}2=\frac{n(k+1)(k+2)(k+4)}2 }
  coefficients.
For any $\bv\in H(\div;T)$, $\Pi_h \b v\in \Lambda_k(T)$ is defined by
\an{ \label{p-j} \ad{ \int_{F_{ij}\subset \partial T} (\Pi_h \b v-\bv) \cdot \b n_{ij}
         p_k dS & = 0 \quad \forall p_k \in P_k(F_{ij}), \\
    \int_T (\Pi_h \b v -\bv )\cdot\bn_1 p_{k-1} d \b x &=0 \quad
                \forall p_{k-1} \in P_{k-1}(T), \\
     \int_{T_i}  (\Pi_h \b v -\bv)\cdot\bn_2 p_{k-1} d \b x &=0 \quad
                \forall p_{k-1} \in P_{k-1}(T_i), \ i=1,...n, \\
     \int_{T_i} ( \Pi_h \b v -\b v )\cdot\bn_3 p_{k-1} d \b x &=0 \quad
                \forall p_{k-1} \in P_{k-1}(T_i), \ i=1,...n, \\
      \int_{F_{ij}\subset T^0} [\Pi_h \b v]\cdot \b n_{ij}p_k  dS &=0
               \quad \forall p_k \in P_k(F_{ij}), \\
    \int_{T_1} \nabla \cdot ( \Pi_h \b v|_{T_i} -\Pi_h \b v|_{T_1} ) p_{k} d \b x &=0 \quad
                \forall p_{k} \in P_k(T_1), \ i=2,...,n,
   } }
where $F_{ij}$ is the $j$-th face triangle of $T_i$ with a fixed normal vector $\bn_{ij}$,
     $\bn_1$ is a unit vector not parallel to any internal face normal $\bn_{ij}$,
    $(\bn_1,\bn_2,\bn_3)$ forms a right-hand orthonormal system,
      $[\cdot]$ denotes the jump on a face triangle, and
    $\Pi_h\b v|_{T_i}$ is understood as a polynomial vector which can be used on
    another tetrahedron $T_1$.
The linear system \eqref{p-j} of equations has the following number of equations,
\a{  &\quad \ (2n+2)\frac{(k+1)(k+2)}2 + (2n+1)\frac{k(k+1)(k+2)}6 \\
     &\quad \ +(n-1) \frac{(k+1)(k+2)}2 + (n-1) \frac{ (k+1)(k+2)(k+3)}6 \\
     &= \frac{n(k+1)(k+2)(k+4)}2,
  } which is exactly the number of coefficients for a $\bv_h$ function in \eqref{v-e}.
Thus we have a square linear system.
The system has a unique solution if and only if the kernel is $\{0\}$.

Let $\bv=0$ in \eqref{p-j}.
Though $\Pi_h\bv$ is a $P_{k+1}$ polynomial,  $\Pi_h\bv\cdot \b n_{ij}$ is a $P_k$ polynomial
   when restricted on $F_{ij}$.  This can be seen by the normal format of plane equation for
   triangle $F_{ij}$.
By the first equation of \eqref{p-j}, $\Pi_h\bv\cdot \b n_{ij}= 0$ on ${F_{ij}}$.
By the sixth equation of \eqref{p-j}, $\nabla \cdot \Pi_h\bv$ is a one-piece polynomial on the whole
   $T$.
Because $\nabla \cdot \Pi_h\bv$ is continuous on inner interface triangles and is a $P_{k}(F_{ij})$
   polynomial on the outer face triangles, by the first five equations in \eqref{p-j},
    we have
\a{ \int_T (\nabla \cdot \Pi_h\bv)^2 d \b x &=\sum_{i=1}^n
      \Big(\int_{T_i} -\Pi_h\bv\cdot \nabla(\nabla \cdot  \Pi_h\bv) d\b x
          +\int_{\partial T_i} \Pi_h\bv\cdot \b n (\nabla \cdot \Pi_h\bv) dS \Big) \\
    &=\sum_{i=1}^n
      \sum_{j=1}^3 \int_{T_i} -(\Pi_h\bv\cdot \bn_j)(\bn_j\cdot \nabla(\nabla \cdot \Pi_h\bv)) d\b x  \\
   &=0.  }
That is,
\an{\label{div0} \nabla \cdot  \Pi_h\bv=0 \quad\text{ on } \ T.  }

Starting from a corner tetrahedron $T_1$, we have its three face triangles, $F_{11}$,
   $F_{12}$ and $F_{13}$, on the boundary of $T$.
The forth face triangle $F_{14}$ of $T_1$ is shared by $T_2$.
By the selection of $\bn_1$,  the normal vector $\bn_{14}=c_1\bn_1+c_2\bn_2+c_3\bn_3$ of
  $F_{14}$ has a non zero $c_1\ne 0$.
   a 2D polynomial $p_k\in P_k(F_{14})$ can be expressed as $p_k(x_2,x_3)$, where we use
   $(x_1,x_2,x_3)$ as the coordinate variables under the system $(\bn_1, \bn_2, \bn_3)$.
Viewing this polynomial as a 3D polynomial, i.e. extending it constantly in $x_1$-direction,
   we have \a{ p_k(x_1,x_2,x_3) = p_k(x_2,x_3), \quad (x_1,x_2,x_3)\in T_1.  }
By \eqref{div0} and the third and fourth equations of \eqref{p-j},  it follows that
  \an{\nonumber 0 &=\int_{T_1} ( \nabla \cdot \Pi_h\bv )  p_k d\b x \\
      \nonumber  &= -\int_{T_1} \Big( (\Pi_h\bv\cdot \bn_1)\partial_{x_1} p_k +
                     (\Pi_h\bv\cdot \bn_2)\partial_{x_2} p_k
            + (\Pi_h\bv\cdot \bn_3)\partial_{x_3} p_k \Big) d\b x \\
      \nonumber  &\quad \    +\int_{F_{14}} (\Pi_h\bv)\cdot \bn_{14} p_k dS\\
      \nonumber  &= -\int_{T_1}  (\Pi_h\bv\cdot \bn_1)\cdot 0 d\b x  + 0+0
              +\int_{F_{14}} (\Pi_h\bv)\cdot \bn_{14} p_k dS\\
     \label{F14}
    &=\int_{F_{14}} (\Pi_h\bv)\cdot \bn_{14} p_k dS \quad\forall p_k\in P_{k}(F_{14}).
   } Next, for any $p_{k-1}\in P_{k-1}(T_1)$, we let $p_k\in P_k(T_1)$ be one of its
   anti-$x_1$-derivative, i.e.,  $\partial_{x_1} p_{k}=p_{k-1}$.
  Thus, by \eqref{div0}, the third and fourth equations of \eqref{p-j} and \eqref{F14}, we get
\an{ \nonumber 0 &=\int_{T_1} \nabla \cdot \Pi_h\bv p_k d\b x \\
    \nonumber    &= -\int_{T_1} \Big((\Pi_h\bv \cdot\bn_1 )  \partial_{x_1} p_{k}  + 0
            + 0 \Big) d\b x
          +\int_{F_{14}} (\Pi_h\bv)\cdot \bn_{14} p_k dS\\
    \label{T1}  &=-\int_{T_1}  (\Pi_h\bv\cdot \bn_1)   p_{k-1}   d\b x  \quad\forall p_{k-1}\in P_{k-1}(T_1).
   }

Continuing work on $T_1$, by $\nabla \cdot \Pi_h\bv=0$,  all $a_{4,ijl}=0$ in \eqref{v-e}, since
   the divergence of each such term is non-zero and independent of the divergence of
   other terms.
Thus $\Pi_h\bv |_{T_i}$ is in $[P_k(T_i)]^d$, instead of $RT_k(T_i)$.
It can be linearly expanded by the three projections on three linearly independent
   directions.
In particular, on a corner tetrahedron $T_1$ we have three outer triangles $F_{1j}$ on
   $\partial T$.
On $T_1$,
\a{ \Pi_h\bv = A \p{ \Pi_h\bv \cdot \bn_{11} \\ \Pi_h\bv \cdot \bn_{12}\\
     \Pi_h\bv \cdot \bn_{13}}= A\p{ p_1 \\ p_2\\p_3 }, }
where $p_1, p_2$ and $ p_3$ are scalar $P_k$ polynomials, and $A$ is a $3\times 3$
   scalar matrix.

By the first equation in \eqref{p-j},  $p_1$ vanishes on $F_{11}$ and
\a{ p_1 = \lambda_1 q_{k-1} \quad \text{ on } \ T_1,
} where $\lambda_1$ is a barycentric coordinate of $T_1$ (which is a linear function assuming   $0$ on $F_{11}$),  and $q_{k-1}$ is a $P_{k-1}(T)$ polynomial.
Let $p_k\in P_k(T)$ be an anti-$x$-derivative of $(\bn_{11})_1 q_{k-1}$, i.e.,
   $(\nabla p_k)_1 = (\bn_{11})_1 q_{k-1}$.
Note that $(\nabla p_k)_2$ and $(\nabla p_k)_3$ can be anything (of $y$ and $z$ functions)
  which result in zero integrals below.
By \eqref{T1} and the third and the fourth equations of \eqref{p-j},
   since $\nabla \cdot \Pi_h \bv=0$, we get
\a{ \int_{T_1} \lambda_1 q_{k-1}^2 d\b x &
     = \int_{T_1} \Pi_h \bv \cdot( \b n_{11} q_{k-1})  d\b x
     =  0. }
Since $\lambda_1>0$ in $T_1$,  we conclude with $q_{k-1}=0$ and $p_1=0$.
Repeating the analysis we get $p_2=p_3=0$ and $\Pi_h\bv =0$ on $T_1$.

Adding the equations \eqref{F14} and \eqref{T1} to \eqref{p-j},  $T_2$ would be a new
  corner tetrahedron with three no-flux boundary triangles.
  Repeating the estimates on $T_1$, it would
   lead $\Pi_h\bv=0$ on $T_2$. Sequentially,  we obtain $\Pi_h\bv=0$ on all $T_i$, i.e.,
   on the whole $T$.

For a $\tau \in H(\div; \Omega)$ and a $v\in P_k(T)$,  we have, by \eqref{p-j}, \eqref{F14}
   and \eqref{T1},
\a{ (\nabla \cdot (\tau -\Pi_h \tau), v)_T &=
     \sum_{i=1}^n \Big(\int_{T_i} (\tau -\Pi_h \tau)\cdot\nabla v d\b x
            + \int_{\partial T_i} (\tau -\Pi_h \tau)\cdot\b n v d S\Big) \\
       &= \sum_{i=1}^n  0 + \int_{\partial T} (\tau -\Pi_h \tau)\cdot\b n v d S\Big) =0.
  }That is, \eqref{key2} holds.

Since $[P_k(T)]^3\subset \Lambda_k$ and $\Pi_h$ is uni-solvent, $\Pi_h \b v = \b v$ for all
  $\bv \in [P_k(T)]^3$.
On one size $1$ $T$, by the finite dimensional norm-equivalence and the
   shape-regularity assumption on sub-triangles,  the interpolation is stable in $L^2(T)$,
  i.e.,
\an{  \|\Pi_h \tau \|_T \le C \|\tau\|_T.   \label{T-stable} }
After a scaling, the constant $C$ in \eqref{T-stable} remains same.
It follows that
\a{ \| \Pi_h\tau -\tau \|^2 &\le C
    \sum_{T\in\mathcal T_h} (\| \Pi_h(\tau -p_{k,T}) \|_T^2 + \| p_{k,T} -\tau \|_T^2 ) \\
      &\le C
    \sum_{T\in\mathcal T_h} (C \| \tau -p_{k,T} \|_T^2 + \| p_{k,T} -\tau \|_T^2 ) \\
      &\le C
    \sum_{T\in\mathcal T_h} h^{2k+2}  | \tau  |_{k+1,T}^2  \\
      &=C h^{2k+2}  | \tau  |_{k+1}^2,
} where $p_{k,T}$ is a $k$-th Taylor polynomial of $\tau$ on $T$.
\end{proof}

\section{Mixed Finite Element Method}
In this section, we develop a mixed finite element method on polytopal mesh by employing our new mixed elements  and obtain optimal order error estimates for the method.
First let $V=H(\div;\Omega)$ and $W=L^2(\Omega)$.

\begin{algorithm}
A mixed finite element method for the problem (\ref{w-mix1})-(\ref{w-mix2})
seeks $(\bq_h,u_h)\in V_h\times W_h$  satisfying
\begin{eqnarray}
(a\bq_h,\bv)-(\nabla\cdot\bv,u_h)&=&\langle g,\bv\cdot\bn\rangle_{\partial\Omega}
\quad\forall\bv\in V_h,\label{mix1}\\
(\nabla\cdot\bq_h,w)&=&(f,w)\quad\forall w\in W_h.\label{mix2}
\end{eqnarray}
\end{algorithm}

We introduce a norm  $\| \bv\|_V$
   for any $\bv\in V$ as follows:
\begin{eqnarray}
\| \bv\|_V^2 &=& \|\bv\|^2+\|\nabla\cdot\bv\|^2. \label{norm}
\end{eqnarray}

\smallskip
\begin{lemma}
There exists a positive constant $\beta$ independent of $h$ such that for all $\rho\in W_h$,
\begin{equation}\label{inf-sup}
\sup_{\bv\in V_h}\frac{(\nabla\cdot\bv,\rho)}{\|\bv\|_V}\ge \beta
\|\rho\|.
\end{equation}
\end{lemma}

\begin{proof}
For any given $\rho\in W_h\subset L^2(\Omega)$, it is known
\cite{bf} that there exists a function
$\tilde\bv\in  V$ such that
\begin{equation}\label{c-inf-sup}
\frac{(\nabla\cdot\tilde\bv,\rho)}{\|\tilde\bv\|_V}\ge C_0\|\rho\|,
\end{equation}
where $C_0>0$ is a constant independent of $h$. By
setting $\bv=\Pi_h\tilde{\bv}\in V_h$ and using (\ref{key3}), we have
\begin{equation}\label{m9}
\|\bv\|_V=\|\Pi_h\tilde{\bv}\|_V\le C\|\tilde{\bv}\|_V.
\end{equation}
Using (\ref{key2}), (\ref{m9}) and (\ref{c-inf-sup}, we have
\begin{eqnarray*}
\frac{|(\nabla\cdot\bv,\rho)|} {\|\bv\|_V}=
\frac{|(\nabla\cdot\Pi_h\tilde{\bv},\rho)|}{\|\bv\|_V}\ge \frac{|(\nabla\cdot\tilde{\bv},\rho)|}{C\|\tilde{\bv}\|_V}\ge \beta\|\rho\|,
\end{eqnarray*}
for a positive constant $\beta$. This completes the proof of the lemma.
\end{proof}

\begin{theorem} Let $(\bq_h, u_h)\in V_h\times W_h$ be the mixed finite element solution of (\ref{mix1})-(\ref{mix2}). Then,
there exists a constant $C$ such that
\begin{equation}\label{err1}
\|\bq-\bq_h\|_V+\|u-u_h\| \le Ch^{k+1}(|\bq|_{k+1}+|u|_{k+1}).
\end{equation}
\end{theorem}
\begin{proof}
Let $\be_h=\Pi_h\bq-\bq_h$ and $\epsilon_h=Q_hu-u_h$, where $Q_h$ is the element-wise defined $L^2$ projection onto $P_k(T)$ on each element $T$.
The differences of (\ref{w-mix1})-(\ref{w-mix2}) and (\ref{mix1})-(\ref{mix2}) imply
\begin{eqnarray}
(a(\bq-\bq_h),\bv)-(\nabla\cdot\bv,u-u_h)&=&0 \quad\quad\quad\forall\bv\in V_h,\label{ee1}\\
(\nabla\cdot (\bq-\bq_h),w)&=&0\quad\quad\quad\forall w\in W_h.\label{ee2}
\end{eqnarray}
By adding $(a\Pi_h\bq_h,\bv)$ to the both sides of (\ref{ee1}) and using the definition of $Q_h$, (\ref{ee1}) becomes
\begin{eqnarray}
(a\be_h,\bv)-(\nabla\cdot\bv,\epsilon_h)&=&(a(\Pi_h\bq-\bq),\bv).\label{new1}
\end{eqnarray}
It follows from (\ref{key2}) and (\ref{ee2}) that for $w\in W_h$
\begin{eqnarray}
(\nabla\cdot\be_h,w)=(\nabla\cdot (\Pi_h\bq-\bq_h),w)=(\nabla\cdot (\bq-\bq_h),w)=0. \label{key5}
\end{eqnarray}
%which implies
%\begin{eqnarray}
%\nabla\cdot\be_h=\nabla\cdot (\Pi_h\bq-\bq_h)=0.\label{key5}
%\end{eqnarray}
Combining (\ref{new1})-(\ref{key5}), we have for all $(\bv,w)\in V_h\times W_h$
\begin{eqnarray}
(a\be_h,\bv)-(\nabla\cdot\bv,\epsilon_h)&=&(a(\Pi_h\bq-\bq),\bv),\label{ee11}\\
(\nabla\cdot \be_h,w)&=&0.\label{ee22}
\end{eqnarray}
Letting $\bv=\be_h$ in (\ref{ee11}) and using (\ref{key5}), we have
\begin{eqnarray*}
(a\be_h,\be_h)=(a(\Pi_h\bq-\bq),\be_h),
\end{eqnarray*}
which gives
\begin{eqnarray}
\|\Pi_h\bq-\bq_h\|_V\le C h^{k+1}|\bq|_{k+1}.\label{m0}
\end{eqnarray}
It follows from (\ref{ee11}) and (\ref{m0}) that for all $\bv\in V_h$
\begin{eqnarray}
(\nabla\cdot\bv,\epsilon_h)\le |(a\be_h,\bv)|+|(a(\Pi_h\bq-\bq),\bv)|\le Ch^{k+1}\|\bq\|_{k+1}\|\bv\|_V.\label{m1}
\end{eqnarray}
The inf-sup condition (\ref{inf-sup}) and the estimate (\ref{m1}) yield
\begin{eqnarray}
\|Q_hu-u_h\|\le  Ch^{k+1}\|\bq\|_{k+1}.\label{m11}
\end{eqnarray}
It follows from (\ref{m0}) and (\ref{m11})
\begin{equation}\label{err11}
\|\Pi_h\bq-\bq_h\|_V+\|Q_hu-u_h\| \le Ch^{k+1}|\bq|_{k+1}.
\end{equation}
The error bound (\ref{err1}) follows from the triangle inequality and (\ref{err11}) and we have proved the theorem.
\end{proof}

To obtain superconvergence for $u_h$, we consider the dual system: seek $(\bpsi,\theta)\in H_0(\div;\Omega)\times L^2(\Omega)$ such that
\begin{eqnarray}
(a\bpsi,\bv)-(\nabla\cdot\bv, \theta)&=&0\quad \forall\bv\in H_0(\div;\Omega),\label{dmix1}\\
(\nabla\cdot \bpsi,w)&=&(Q_hu-u_h,w)\quad\forall w\in L^2(\Omega).\label{dmix2}
\end{eqnarray}
Assume that the following regularity holds
\begin{equation}\label{reg}
\|\bpsi\|_1+\|\theta\|_1\le C\|Q_hu-u_h\|.
\end{equation}

\begin{theorem} Let $(\bq_h, u_h)\in V_h\times W_h$ be the mixed finite element solution of (\ref{mix1})-(\ref{mix2}). Assume that 
(\ref{reg}) holds true. Then,
there exists a constant $C$ such that
\begin{equation}\label{err10}
\|Q_h u-u_h\| \le Ch^{k+2}(|\bq|_{k+1}+|u|_{k+1}).
\end{equation}
\end{theorem}
\begin{proof}
Letting $w=Q_hu-u_h$ in (\ref{dmix2})  and using (\ref{key2}), (\ref{ee1}), (\ref{dmix1}), (\ref{ee2}), (\ref{err1}) and (\ref{reg}), we have
\begin{eqnarray*}
\|Q_hu-u_h\|^2&=&(\nabla\cdot\bpsi, Q_hu-u_h)\\
&=&(\nabla\cdot\Pi_h\bpsi, Q_hu-u_h)\\
&=&(\Pi_h\bpsi, a(\bq-\bq_h))\\
&=&(\Pi_h\bpsi-\bpsi, a(\bq-\bq_h))+(\bpsi, a(\bq-\bq_h))\\
&=&(\Pi_h\bpsi-\bpsi, a(\bq-\bq_h))+(\nabla\cdot(\bq-\bq_h),\theta)\\
&=&(\Pi_h\bpsi-\bpsi, a(\bq-\bq_h))+(\nabla\cdot(\bq-\bq_h),\theta-Q_h\theta)\\
&\le& Ch^{k+2}\|\bq\|_{k+1}\|Q_hu-u_h\|,
\end{eqnarray*}
which implies (\ref{err10}) and we have proved the theorem.
\end{proof}

\section{Numerical Example}

We solve problem \eqref{mixed1}--\eqref{bc1} on the unit square domain with the exact
  solution
\an{ \label{s-1} \b q=\p{\pi\sin(\pi y)\cos(\pi x)\\ \pi\sin(\pi x)\cos(\pi y)},
    \quad u=\sin(\pi x)\sin(\pi y).
  } We first use quadrilateral grids.  To avoid asymptotic parallelograms under
   nested refinements,  we use fixed types of quadrilaterals in our multi-level
   grids, shown in Figure \ref{g-4}.
We list the computational results in Table \ref{t1}.
As proved,   we have one order of super-convergence for both
   $u_h$ and $\b q_h$.

\begin{figure}[htb]\begin{center}
\includegraphics[width=1.4in]{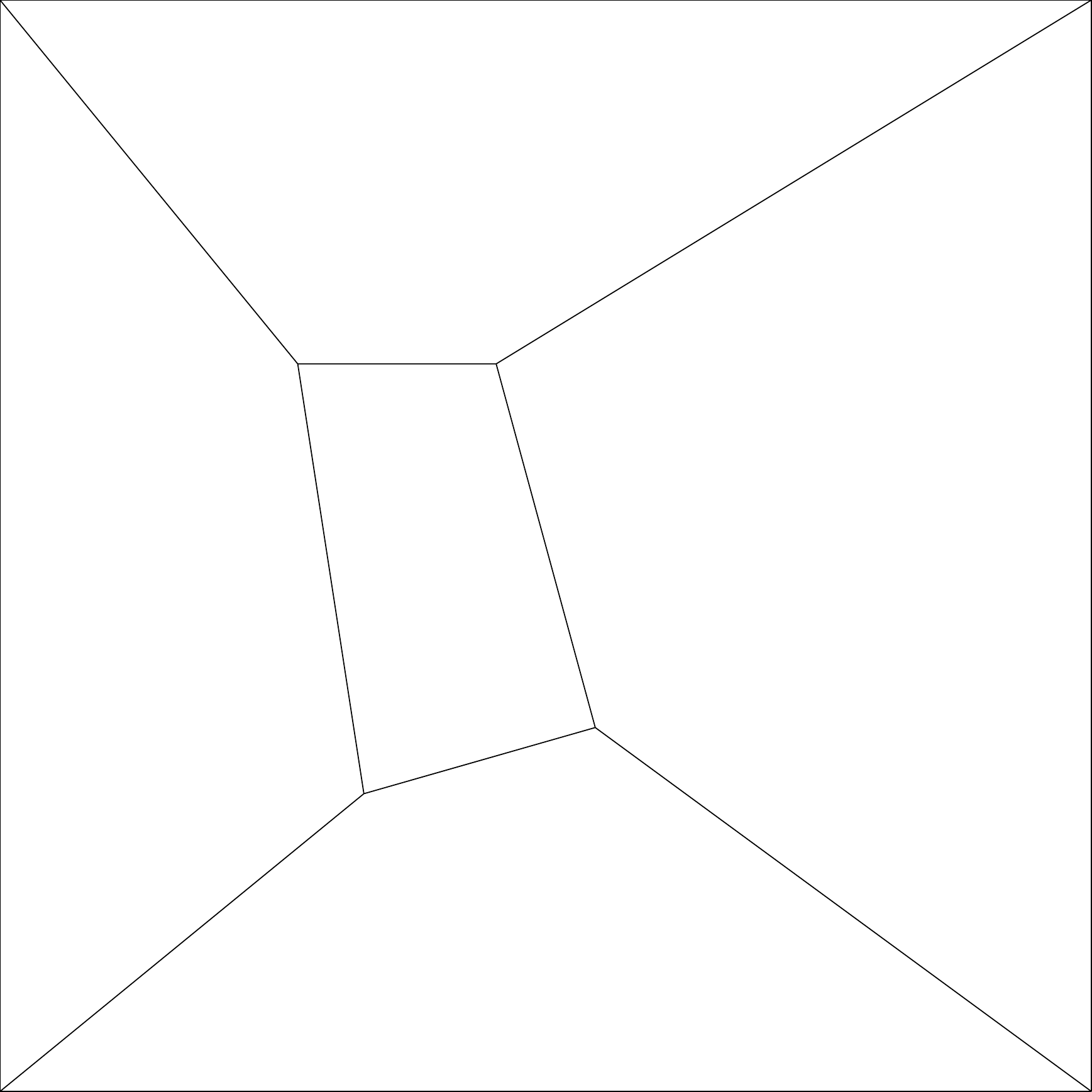} \
\includegraphics[width=1.4in]{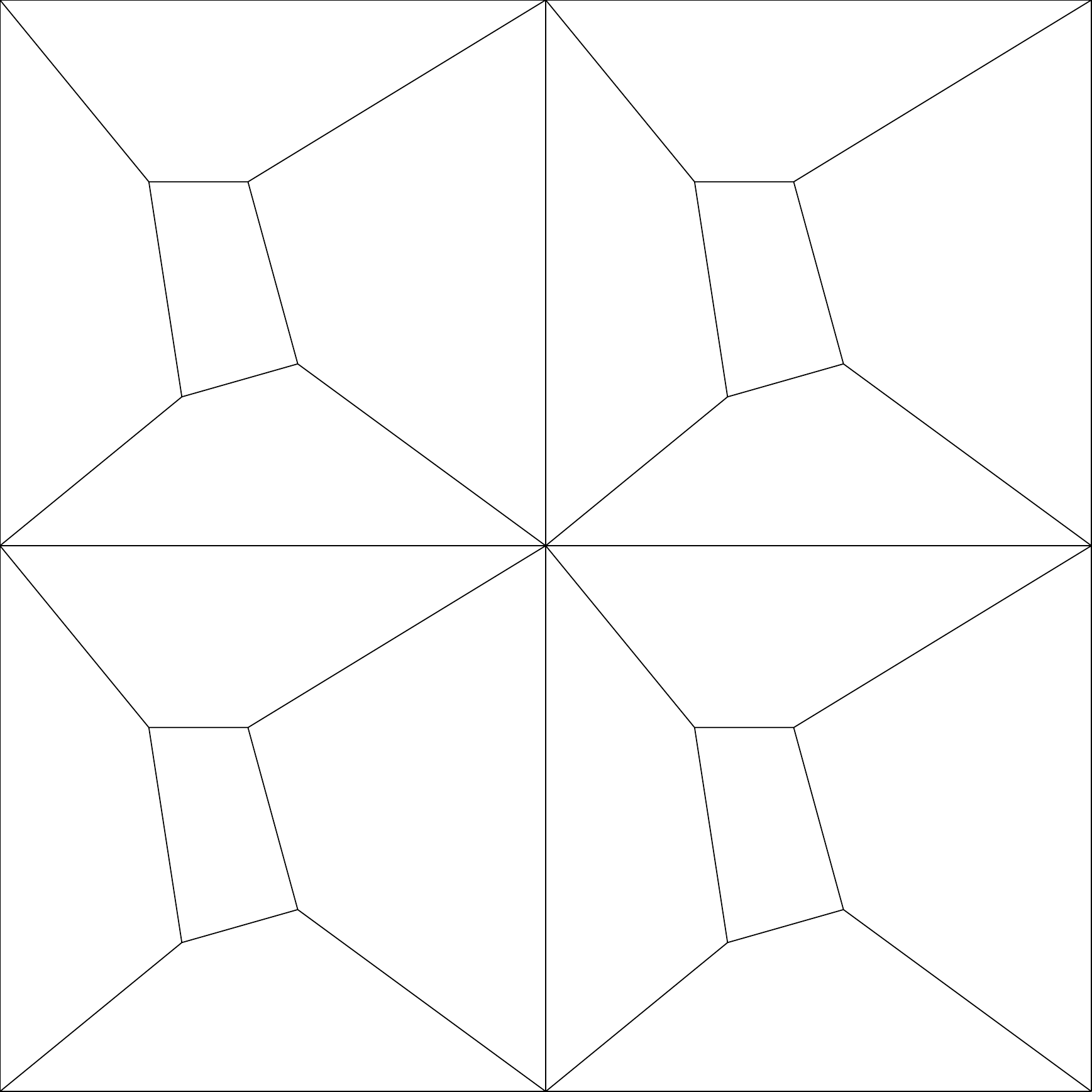} \
\includegraphics[width=1.4in]{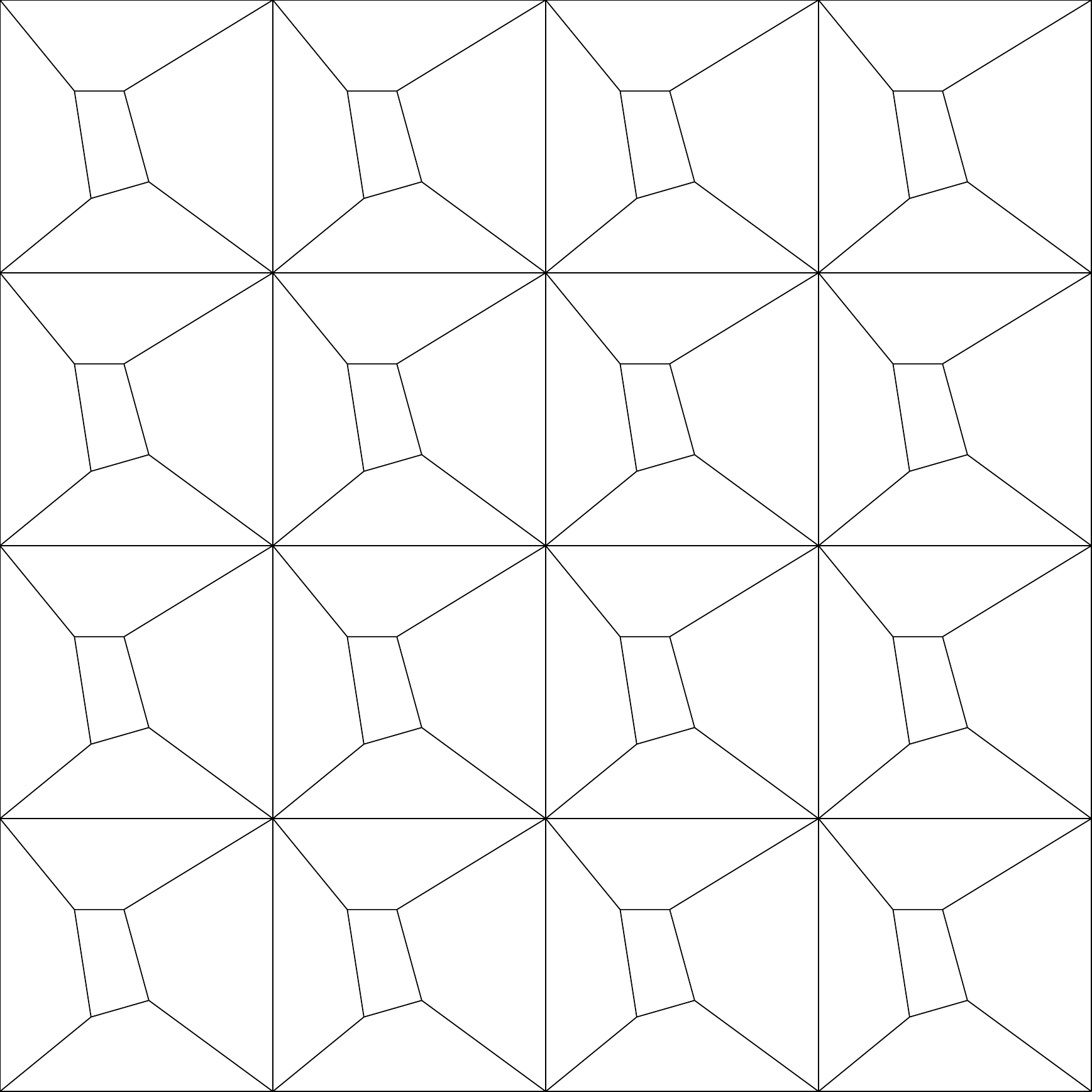}

\caption{The first three levels of grids, for Table \ref{t1}.  }
\label{g-4}
\end{center}
\end{figure}

\begin{table}[h!]
  \centering   \renewcommand{\arraystretch}{1.05}
  \caption{ Error profiles and convergence rates on grids shown in Figure \ref{g-4} for \eqref{s-1}. }
\label{t1}
\begin{tabular}{c|cc|cc}
\hline
level & $\|Q_h u-  u_h \|_0 $  &rate &  $\|\Pi_h \bq- \bq_h \|_V $ &rate   \\
\hline
 &\multicolumn{4}{c}{by the $\Lambda_0$-$P_0$ mixed element} \\ \hline
 6&   0.1464E-03 &  2.00&   0.5185E-01 &  1.00 \\
 7&   0.3660E-04 &  2.00&   0.2593E-01 &  1.00 \\
 8&   0.9151E-05 &  2.00&   0.1296E-01 &  1.00 \\
\hline
 &\multicolumn{4}{c}{by the $\Lambda_1$-$P_1$ mixed element} \\ \hline
 6&   0.1072E-05 &  3.00&   0.4103E-03 &  2.00 \\
 7&   0.1340E-06 &  3.00&   0.1025E-03 &  2.00 \\
 8&   0.1674E-07 &  3.00&   0.2563E-04 &  2.00 \\
 \hline
 &\multicolumn{4}{c}{by the $\Lambda_2$-$P_2$ mixed element} \\ \hline
 5&   0.5704E-06 &  4.00&   0.1878E-03 &  3.00 \\
 6&   0.3567E-07 &  4.00&   0.2349E-04 &  3.00 \\
 7&   0.2231E-08 &  4.00&   0.2937E-05 &  3.00 \\
\hline
 &\multicolumn{4}{c}{by the $\Lambda_3$-$P_3$ mixed element} \\ \hline
 3&   0.1403E-05 &  5.84&   0.1559E-03 &  4.88 \\
 4&   0.2765E-07 &  5.66&   0.5969E-05 &  4.71 \\
 5&   0.6808E-09 &  5.34&   0.2837E-06 &  4.39 \\
 \hline
\end{tabular}%
\end{table}%

\begin{figure}[htb]\begin{center}
\includegraphics[width=1.4in]{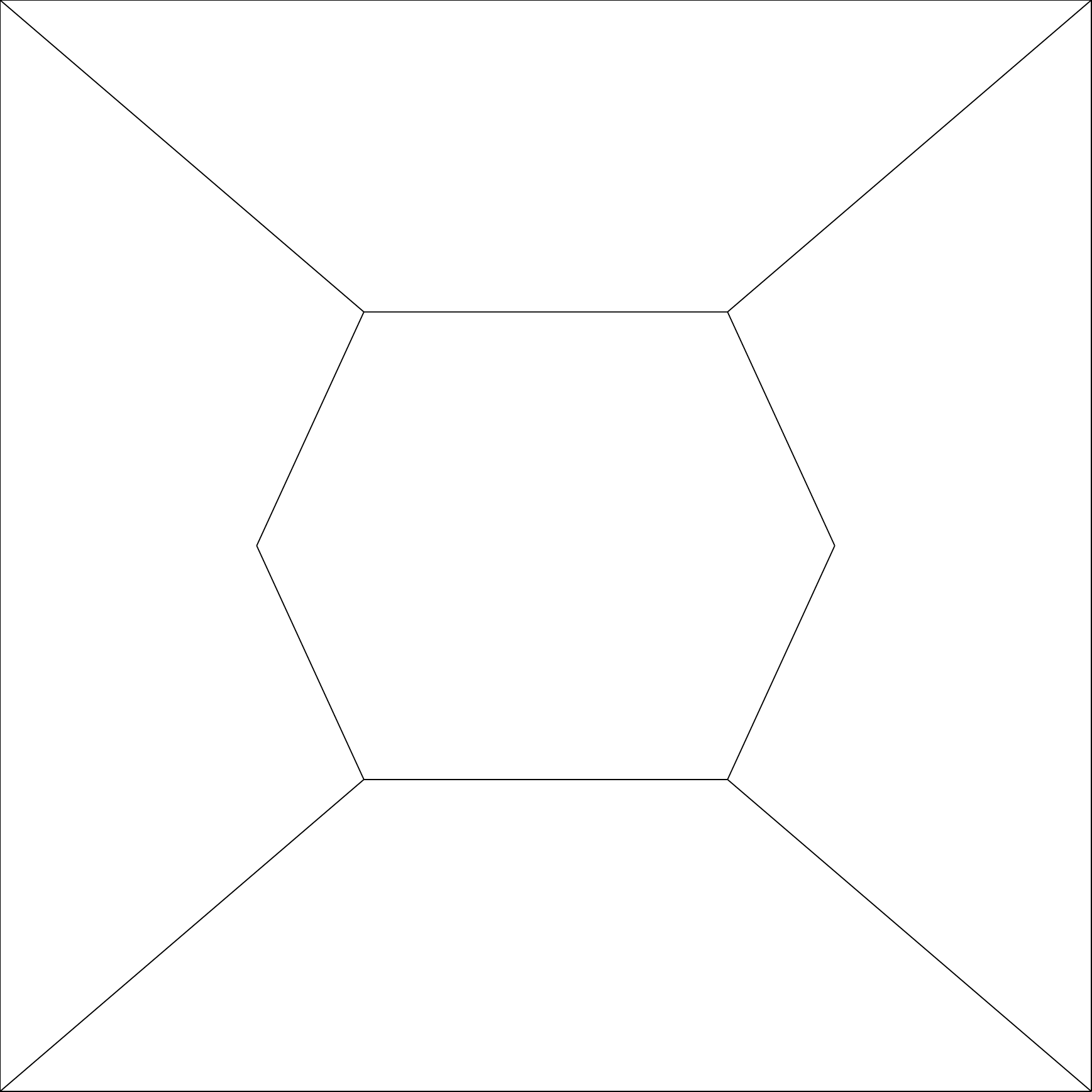} \
\includegraphics[width=1.4in]{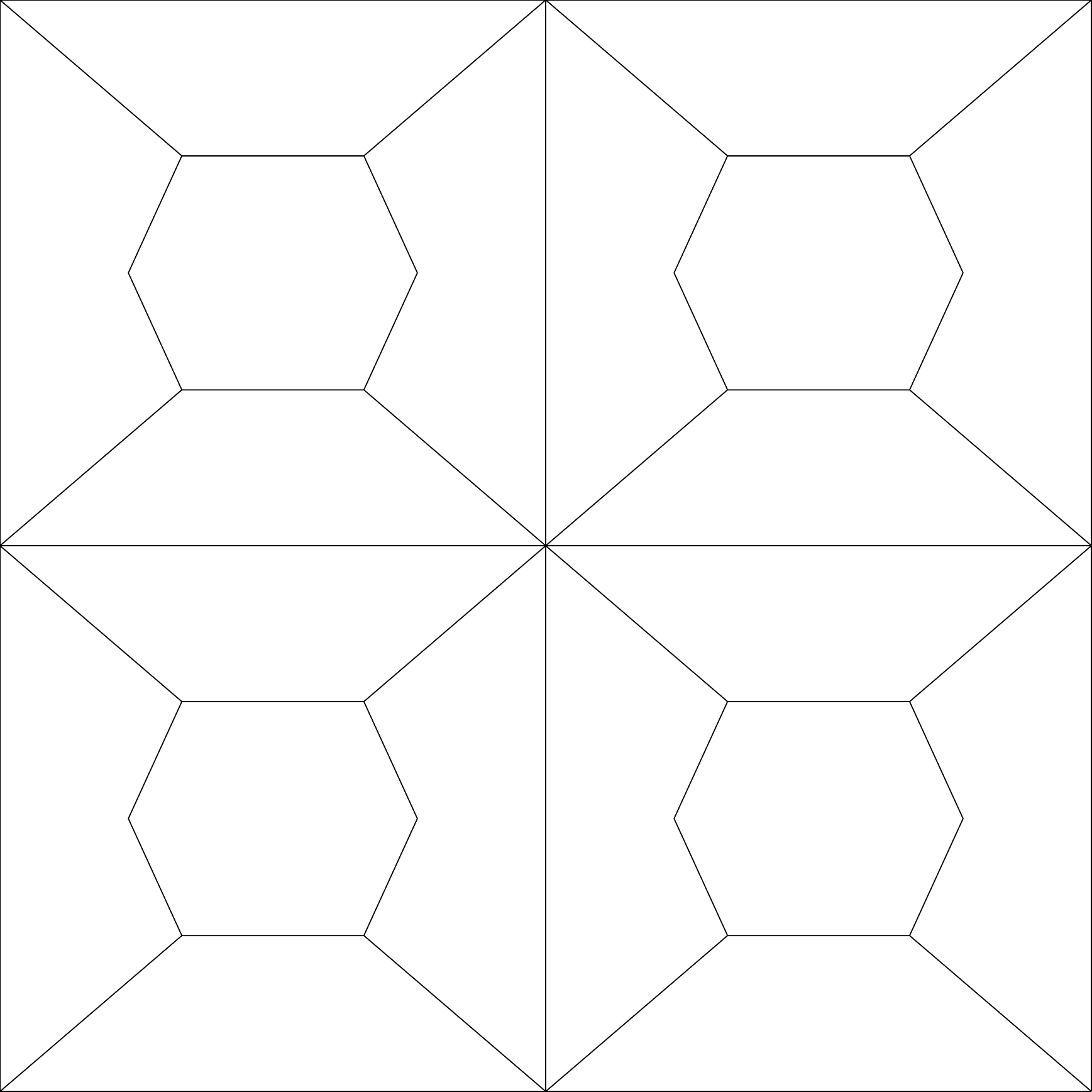} \
\includegraphics[width=1.4in]{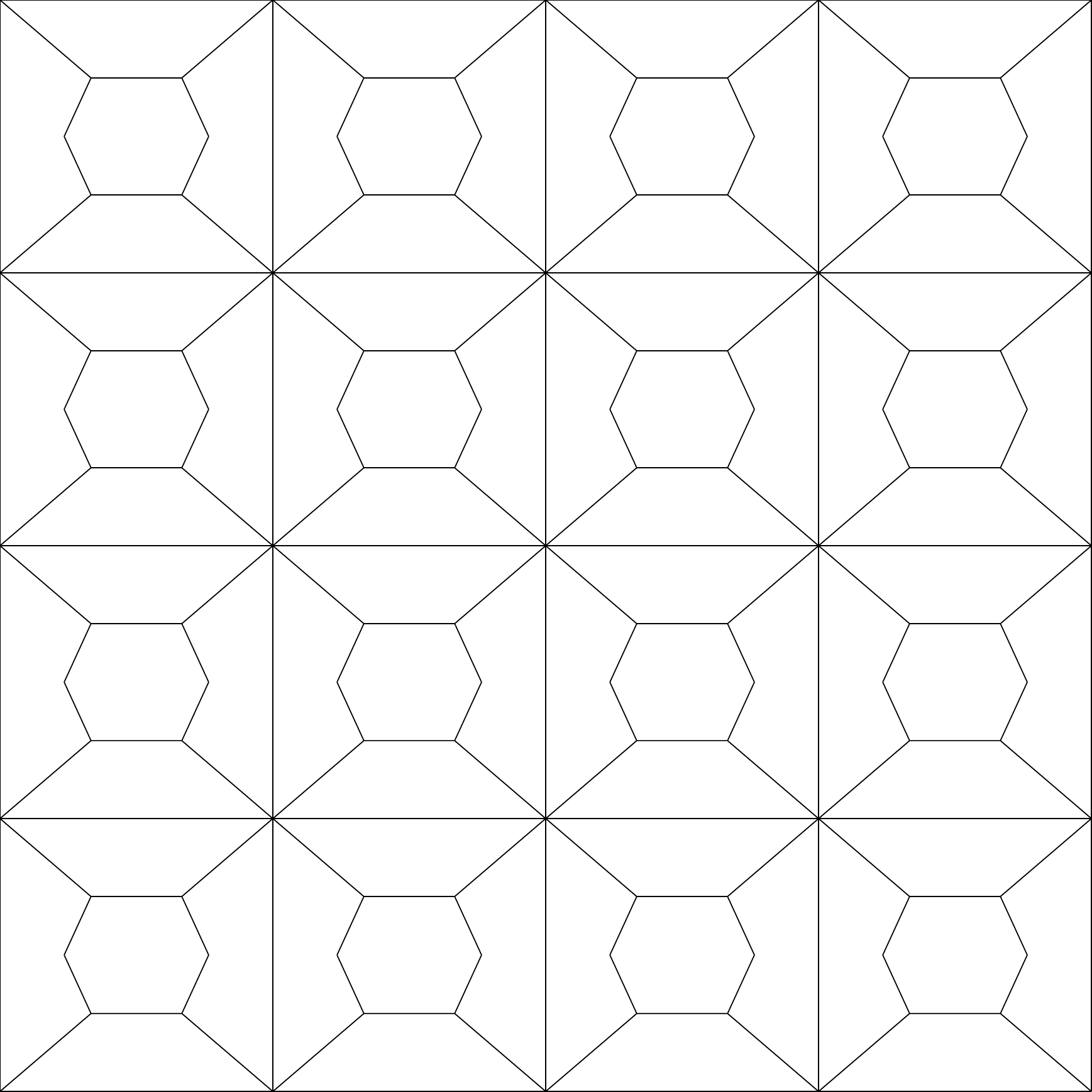}

\caption{The first three levels of quadrilateral-hexagon grids, for Table \ref{t2}.  }
\label{g-6}
\end{center}
\end{figure}

Next we solve the same problem \eqref{s-1} on a type of grids with quadrilaterals and hexagons,
 shown in Figure \ref{g-6}.
We list the result of computation in Table \ref{t2} where we obtain
   one order of superconvergence in all cases.

\begin{table}[h!]
  \centering   \renewcommand{\arraystretch}{1.05}
  \caption{ Error profiles and convergence rates on grids shown in Figure \ref{g-6} for \eqref{s-1}. }
\label{t2}
\begin{tabular}{c|cc|cc}
\hline
level & $\|Q_h u-  u_h \|_0 $  &rate &  $\|\Pi_h \bq- \bq_h \|_V $ &rate   \\
\hline
 &\multicolumn{4}{c}{by the $\Lambda_0$-$P_0$ mixed element} \\ \hline
 6&   0.1523E-03 &  2.00&   0.5282E-01 &  1.00 \\
 7&   0.3808E-04 &  2.00&   0.2641E-01 &  1.00 \\
 8&   0.9520E-05 &  2.00&   0.1321E-01 &  1.00 \\
\hline
 &\multicolumn{4}{c}{by the $\Lambda_1$-$P_1$ mixed element} \\ \hline
 6&   0.1015E-05 &  3.00&   0.3958E-03 &  2.00 \\
 7&   0.1269E-06 &  3.00&   0.9893E-04 &  2.00 \\
 8&   0.1586E-07 &  3.00&   0.2473E-04 &  2.00 \\
 \hline
 &\multicolumn{4}{c}{by the $\Lambda_2$-$P_2$ mixed element} \\ \hline
 5&   0.8069E-07 &  4.01&   0.2283E-04 &  3.05 \\
 6&   0.5038E-08 &  4.00&   0.2830E-05 &  3.01 \\
 7&   0.3149E-09 &  4.00&   0.3530E-06 &  3.00 \\
\hline
 &\multicolumn{4}{c}{by the $\Lambda_3$-$P_3$ mixed element} \\ \hline
 3&   0.9106E-06 &  5.72&   0.1176E-03 &  4.83 \\
 4&   0.2080E-07 &  5.45&   0.4844E-05 &  4.60 \\
 5&   0.5735E-09 &  5.18&   0.2481E-06 &  4.29  \\
 \hline
\end{tabular}%
\end{table}%

\begin{figure}[h!]
\begin{center}
 \setlength\unitlength{1pt}
    \begin{picture}(320,118)(0,3)
    \put(0,0){\begin{picture}(110,110)(0,0)
       \multiput(0,0)(80,0){2}{\line(0,1){80}}  \multiput(0,0)(0,80){2}{\line(1,0){80}}
       \multiput(0,80)(80,0){2}{\line(1,1){20}} \multiput(0,80)(20,20){2}{\line(1,0){80}}
       \multiput(80,0)(0,80){2}{\line(1,1){20}}  \multiput(80,0)(20,20){2}{\line(0,1){80}}
    \put(80,0){\line(-1,1){80}}% \put(80,0){\line(1,5){20}}\put(80,80){\line(-3,1){60}}
      \end{picture}}
    \put(110,0){\begin{picture}(110,110)(0,0)
       \multiput(0,0)(40,0){3}{\line(0,1){80}}  \multiput(0,0)(0,40){3}{\line(1,0){80}}
       \multiput(0,80)(40,0){3}{\line(1,1){20}} \multiput(0,80)(10,10){3}{\line(1,0){80}}
       \multiput(80,0)(0,40){3}{\line(1,1){20}}  \multiput(80,0)(10,10){3}{\line(0,1){80}}
    \put(80,0){\line(-1,1){80}}% \put(80,0){\line(1,5){20}}\put(80,80){\line(-3,1){60}}
       \multiput(40,0)(40,40){2}{\line(-1,1){40}}
       %  \multiput(80,40)(10,-30){2}{\line(1,5){10}}
       %  \multiput(40,80)(50,10){2}{\line(-3,1){30}}
      \end{picture}}
    \put(220,0){\begin{picture}(110,110)(0,0)
       \multiput(0,0)(20,0){5}{\line(0,1){80}}  \multiput(0,0)(0,20){5}{\line(1,0){80}}
       \multiput(0,80)(20,0){5}{\line(1,1){20}} \multiput(0,80)(5,5){5}{\line(1,0){80}}
       \multiput(80,0)(0,20){5}{\line(1,1){20}}  \multiput(80,0)(5,5){5}{\line(0,1){80}}
    \put(80,0){\line(-1,1){80}}% \put(80,0){\line(1,5){20}}\put(80,80){\line(-3,1){60}}
       \multiput(40,0)(40,40){2}{\line(-1,1){40}}
       %  \multiput(80,40)(10,-30){2}{\line(1,5){10}}
       %  \multiput(40,80)(50,10){2}{\line(-3,1){30}}

       \multiput(20,0)(60,60){2}{\line(-1,1){20}}   \multiput(60,0)(20,20){2}{\line(-1,1){60}}
       %  \multiput(80,60)(15,-45){2}{\line(1,5){5}} \multiput(80,20)(5,-15){2}{\line(1,5){15}}
       %  \multiput(20,80)(75,15){2}{\line(-3,1){15}}\multiput(60,80)(25,5){2}{\line(-3,1){45}}
      \end{picture}}

    \end{picture}
    \end{center}
\caption{  The first three levels of wedge grids used in Table \ref{t3}. }
\label{grid3}
\end{figure}
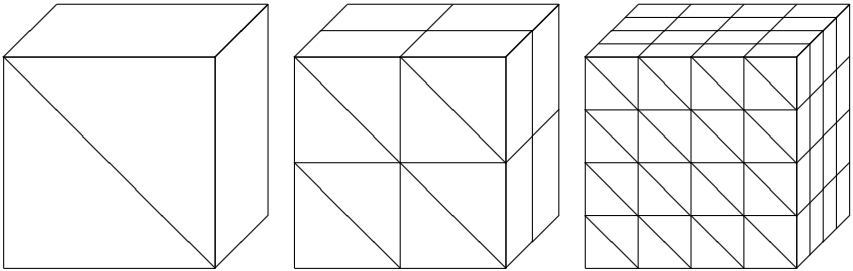

We solve 3D problem \eqref{mixed1}--\eqref{bc1} on the unit cube domain
  $\Omega=(0,1)^3$ with the exact
  solution
\an{ \label{s-2} \ad{ \b q&=\p{2^8(1-2x)^2 (y-y^2)(z-z^2)\\
                         2^8(x-x^2)^2 (1-2y)(z-z^2) \\
                          2^8(x-x^2)^2 (y-y^2)(1-2z) }, \\
    \quad u&=2^8(x-x^2)^2 (y-y^2)(z-z^2). }
  }
 Here we use a uniform wedge-type (polyhedron with 2 triangle faces and 3 rectangle faces)
   grids,  shown in Figure \ref{grid3}.   Here each wedge is subdivided in to three
   tetrahedrons with three rectangular faces being cut, when defining piecewise $RT_k$
  element $\Lambda_k$.
 The results are listed in Table \ref{t3},  confirming the theory.

\begin{table}[h!]
  \centering   \renewcommand{\arraystretch}{1.05}
  \caption{ Error profiles and convergence rates on grids shown in Figure \ref{grid3} for \eqref{s-2}. }
\label{t3}
\begin{tabular}{c|cc|cc}
\hline
level & $\|Q_h u-  u_h \|_0 $  &rate &  $\|\Pi_h \bq- \bq_h \|_V $ &rate   \\
\hline
 &\multicolumn{4}{c}{by the 3D $\Lambda_0$-$P_0$ mixed element} \\ \hline
 5&     0.0044197&2.0&     0.5802877&1.0 \\
 6&     0.0011145&2.0&     0.2909994&1.0 \\
 7&     0.0002793&2.0&     0.1456072&1.0 \\
\hline
 &\multicolumn{4}{c}{by the  3D $\Lambda_1$-$P_1$ mixed element} \\ \hline
 4&     0.0049106&2.9&     0.5234688&2.0 \\
 5&     0.0006228&3.0&     0.1317047&2.0 \\
 6&     0.0000782&3.0&     0.0329830&2.0 \\
 \hline
 &\multicolumn{4}{c}{by the 3D  $\Lambda_2$-$P_2$ mixed element} \\ \hline
 4&     0.0004943&4.0&     0.1005523&3.0 \\
 5&     0.0000310&4.0&     0.0126031&3.0 \\
 6&     0.0000019&4.0&     0.0015765&3.0 \\
\hline
 &\multicolumn{4}{c}{by the  3D $\Lambda_3$-$P_3$ mixed element} \\ \hline
 3&     0.0006668&5.0&     0.1986805&3.9 \\
 4&     0.0000207&5.0&     0.0125641&4.0 \\
 5&     0.0000006&5.0&     0.0007875&4.0 \\ \hline
\end{tabular}%
\end{table}%

\end{document}